\documentclass[a4paper,12pt]{amsart}
\usepackage[T1]{fontenc}    
\usepackage[utf8]{inputenc}
\usepackage{enumerate}
\usepackage{amssymb}
\usepackage{graphics,xy,epsfig,picture}
\usepackage{fullpage}
\usepackage{color}
\usepackage{pgf,tikz}
\usepackage{mathrsfs}
\usetikzlibrary{arrows}

\def\RR{\mathbb R}
\def\R{\mathbb R}

\def\ZZ{\mathbb Z}
\def\QQ{\mathbb Q}

\DeclareMathOperator{\sgn}{sgn}

\newcommand{\set}[1]{\left\lbrace #1\right\rbrace}
\providecommand{\abs}[1]{\left\lvert#1\right\rvert}
\providecommand{\norm}[1]{\left\lVert#1\right\rVert}

\newcommand{\remove}[1]{ }
\newcommand{\qtq}[1]{\quad\text{#1}\quad}

\newtheorem{theorem}{Theorem}[section]
\newtheorem{proposition}[theorem]{Proposition}
\newtheorem{lemma}[theorem]{Lemma}

\newtheorem*{TA}{Theorem A}
\newtheorem*{TB}{Theorem B}

\theoremstyle{remark}
\newtheorem{remark}[theorem]{Remark}
\newtheorem{remarks}[theorem]{Remarks}
\newtheorem{example}[theorem]{Example}

\numberwithin{equation}{section}

\begin{document}
\title{Moving and oblique observations of beams and plates}

\author{Philippe Jaming}
\address{Univ. Bordeaux, IMB, UMR 5251, F-33400 Talence, France.
CNRS, IMB, UMR 5251, F-33400 Talence, France.}
\email{Philippe.Jaming@math.u-bordeaux.fr}

\author{Vilmos Komornik} 
\address{College of Mathematics and Computational Science, Shenzhen Uni- versity, Shenzhen 518060, People’s Republic of China, and Département de mathématique\\
         Université de Strasbourg\\
         7 rue René Descartes\\
         67084 Strasbourg Cedex, France}
\email{komornik@math.unistra.fr}
\thanks{The second author was supported by the grant NSFC No. 11871348.}

\begin{abstract}
We study the observability of the one-dimensional Schr\"odinger equation and of the beam and plate equations by moving or oblique observations.
Applying different versions and adaptations of Ingham's theorem on nonharmonic Fourier series, we obtain various observability and non-observability theorems. 
Several open problems are also formulated at the end of the paper.
\end{abstract}

\thanks{Version of 2019-03-14-c}

\maketitle

\section{Introduction}

Fourier series methods have been applied for a long time in control theory \cite{Fattorini-Russell-1974, Krabs-1992, Russell-1978,Ball-Slemrod-1979}.
Since Haraux \cite{Haraux-1989} recognized the usefulness of a classical theorem of Ingham \cite{Ingham-1936} in this context, many new results have been obtained by applying multiple variants of Ingham's theorem \cite{Baiocchi-Komornik-Loreti-1999,Baiocchi-Komornik-Loreti-2002,Beurling,Kahane-1962,Komornik-Loreti-2014}.

The purpose of this paper is to investigate the observability of beams and plates by moving or oblique observations.

Moving point observability theorems for parabolic and hyperbolic equations have been obtained earlier by Khapalov by different methods \cite{Khapalov2,Khapalovbook}.

Another motivation for this paper was the following recent result of the first author with K. Kellay \cite{Jaming-Kellay-2018}:

\begin{theorem}
Let $\mu$ be a bounded measure on $\R^2$ and let $u=\widehat{\mu}$ be its Fourier transform. Assume that
$u$ is a solution of the Schr\"odinger equation $\partial_t u(t,x)+i\partial_x^2 u(t,x)=0$ on $\R_+\times\R$ and assume that,
for some $a\not=b\in\R$, $u(t,at)=u(t,bt)=0$ for every $t>0$ then $u=0$.
\end{theorem}

In other words, a solution of the Schr\"odinger equation is uniquely determined by its value in two moving points $x=at$ and $x=bt$, $t>0$.
The proof however does not provide any quantitative estimate on $u$ from its values on these points.

We first consider the one-dimensional Schr\"odinger equation 
$u_t+iu_{xx}=0$ in a bounded interval $I$ with periodic boundary conditions and initial data $u_0\in L^2(I)$. 
We prove among other things the  observability relations
\begin{equation*}
\int_0^T\abs{u(t,at)}^2\,\mbox{d}t\asymp \norm{u_0}_{L^2(I)}^2
\end{equation*}
for all non-integer real numbers $a$ and for all $T>0$.
(See the beginning of the next section for the notations.)
On the other hand, the relations 
\begin{equation*}
\sum_{i=1}^m\int_0^T\abs{u(t,a_it)}^2\,\mbox{d}t\asymp \norm{u_0}_{L^2(I)}^2
\end{equation*}
fail for any choice of finitely many integers $a_i$ and for any $T>0$.

Here and in the sequel the notation $A\asymp B$ means that $c_1 A\le B\le c_2 A$ with some positive constants $c_1, c_2$ that do not depend on the particular choice of the initial data.

Next we carry over a similar study for the one-dimensional beam equation $u_{tt}+u_{xxxx}=0$  in a bounded interval $I$ with periodic boundary conditions and initial data $u_0\in L^2(I)$, $u_1\in H^{-2}(I)$. 
For example, we have
\begin{equation*}
\int_0^T\abs{u(t,at)}^2\,\mbox{d}t\asymp \norm{u_0}_{L^2(I)}^2+\norm{u_1}_{H^{-2}(I)}^2
\end{equation*}
if and only if the circle centered in $(\frac{-a}{2},\frac{a}{2})$ and passing through the origin contains no other points with integer coordinates.
This is the case whenever $a$ is irrational.
Furthermore, we give a necessary and sufficient geometric condition for the validity of the estimates 
\begin{equation*}
\int_0^T\abs{u(t,a_1t)}^2\,\mbox{d}t
+\int_0^T\abs{u(t,a_2t)}^2\,\mbox{d}t
\asymp \norm{u_0}_{L^2(I)}^2+\norm{u_1}_{H^{-2}(I)}^2
\end{equation*}
in case of two given numbers $a_1, a_2$.
It remains an open question whether there exist exceptional cases indeed.

Finally we consider vibrating rectangular plates. 
Improving several earlier theorems given in \cite{Haraux-1989,Jaf1988,Jaffard-1990,
Komornik-1992}, it was shown in \cite{Komornik-Loreti-2014} that these plates may be observed on an arbitrarily small segment which is parallel to one of the sides of the rectangle.
Using a different tool we prove that the observability still holds for oblique segments. 

The paper is organized as follows. 
In Section \ref{s2} we recall some Ingham type theorems that we need in the subsequent proofs.
Section \ref{s3} is then devoted to the one-dimensional Schr\"odinger equation while
Section \ref{s4} is devoted to the one-dimensional beam equation and Section \ref{s5}
to vibrating rectangular plates.
We end the paper with a list of open questions related to the problems studied here.

\section{A review of Ingham type inequalities}\label{s2}

Ingham type inequalities play a central role is this study. We therefore devote this section to summarize the results we use.

If $I$ is an interval of length $\abs{I}=2\pi$, then Parseval's equality 
\begin{equation*}
\frac{1}{\abs{I}}\int_I\abs{\sum_{k\in\ZZ}c_ke^{ikx}}^2\,\mbox{d}x=\sum_{k\in\ZZ}\abs{c_k}^2
\end{equation*}
holds for all square summable sequences $(c_k)$ of complex numbers. 
This equality remains valid if the length of $I$ is a positive multiple of $2\pi$.
It follows by an elementary argument that if $2k\pi<\abs{I}<(2k+2)\pi$ for some nonnegative integer $k$, then
\begin{equation*}
2k\pi\sum_{k\in\ZZ}\abs{c_k}^2\le\int_I\abs{\sum_{k\in\ZZ}c_ke^{ikx}}^2\,\mbox{d}x\le (2k+2)\pi\sum_{k\in\ZZ}\abs{c_k}^2
\end{equation*}
for all square summable sequences $(c_k)$, and the constants $2k\pi, (2k+2)\pi$ are the best possible here.
Hence 
\begin{equation*}
\int_I\abs{\sum_{k\in\ZZ}c_ke^{ikx}}^2\,\mbox{d}x\asymp \sum_{k\in\ZZ}\abs{c_k}^2
\end{equation*}
for every bounded interval $I$ of length $\ge 2\pi$.
Here and in the sequel we use the notation $A\ll B$ if there exists a positive constant $\alpha$ such that $A\le\alpha B$ for all sequences $(c_k)$, and $A\asymp B$ if $A\ll B$ and $B\ll A$.

Ingham \cite{Ingham-1936} proved an important generalization of the last relation.
A set $\Lambda$ of real numbers is called \emph{uniformly separated} if 
\begin{equation}\label{21}
\gamma(\Lambda):=\inf\set{\abs{\lambda_1-\lambda_2}\ :\ \lambda_1,\lambda_2\in\Lambda\qtq{and}\lambda_1\ne\lambda_2}>0;
\end{equation}
then $\gamma(\Lambda)$ is called the \emph{uniform gap} of $\Lambda$.
For example, $\ZZ$ is uniformly separated with $\gamma(\ZZ)=1$.
Note that the empty set and the one-point sets are uniformly separated with $\gamma(\Lambda)=\infty$.

\begin{TA}[Ingham]
Let $\Lambda\subset\RR$ be a uniformly separated set. 
\begin{enumerate}[\upshape (i)]
\item $\sum_{\lambda\in\Lambda}c_\lambda e^{i\lambda x}$ is a well-defined locally square summable function on $\RR$ for every square summable sequence $(c_\lambda)$.
\item The direct inequality 
\begin{equation*}
\int_I\abs{\sum_{\lambda\in\Lambda}c_\lambda e^{i\lambda x}}^2\,\mbox{d}x\ll \sum_{\lambda\in\Lambda}\abs{c_\lambda}^2
\end{equation*}
holds for every bounded interval $I$. 
\item The inverse inequality 
\begin{equation*}
\sum_{\lambda\in\Lambda}\abs{c_\lambda}^2\ll \int_I\abs{\sum_{\lambda\in\Lambda}c_ke^{i\lambda x}}^2\,\mbox{d}x
\end{equation*}
holds for every bounded interval $I$ of length $>\frac{2\pi}{\gamma(\Lambda)}$.
\end{enumerate}
\end{TA}

\begin{remark}\label{r21}
If $\Lambda$ is not uniformly separated, but it is the union of finitely many, say $m$ uniformly separated sets, then a simple application of the inequality
\begin{equation*}
(x_1+\cdots+x_m)^2\le m(x_1^2+\cdots+x_m^2)
\end{equation*}
shows that the direct inequality still holds.
\end{remark}

The condition $\abs{I}>\frac{2\pi}{\gamma(\Lambda)}$ is the best uniform condition for all uniformly separated sets, but it can be weakened  for individual 
uniformly separated sets.
We illustrate this by recalling from \cite{Haraux-1989} the following

\begin{TB}[Haraux]
If $\Lambda\subset\RR$ is a uniformly separated set and $F\subset\Lambda$ is a finite subset, then the inverse inequality of Theorem A holds under the condition $\abs{I}>\frac{2\pi}{\gamma(\Lambda\setminus F)}$.
\end{TB}

\begin{example}\label{e22}
The square numbers $0,1,4,9,\ldots$ form a uniformly separated set $\Lambda$ with $\gamma(\Lambda)=1$. 
If $F=\set{0,1,\ldots,m-1}$ for some positive integer $m$, then  $\gamma(\Lambda\setminus F)=2m+1$, so that the inverse inequality holds under the condition $\abs{I}>\frac{2\pi}{2m+1}$. 
Since $m$ may be chosen arbitrarily large, hence the inverse inequality holds for all non-degenerated bounded intervals.
\end{example}

\begin{remark}\label{r23}
The proof of Theorem B shows that the direct and inverse inequalities remain valid under the same assumptions for more general sums of the form
\begin{equation*}
\sum_{\lambda\in\Lambda'}c_{\lambda'}xe^{i\lambda' x}+
\sum_{\lambda\in\Lambda}c_\lambda e^{i\lambda x}
\end{equation*}
where $\Lambda'$ is some finite subset of $\Lambda$; see \cite[Theorem 4.5]{Komornik-Loreti-2005}.
\end{remark}

For a particular uniformly separated set the optimal condition for the inverse inequality  has been determined by Beurling \cite{Beurling}; see also \cite{Baiocchi-Komornik-Loreti-2002} for a generalization to weakly separated sets.

\section{One-dimensional Schr\"odinger equation}\label{s3}

We consider the one-dimensional Schr\"odinger equation on a bounded interval with periodic boundary condition.
Up to an affine change of variable, we may assume that the interval is $(0,2\pi)$. Thus we consider the following system
\begin{equation}\label{31}
\begin{cases}
u_t+iu_{xx}=0&\mbox{in }\RR\times(0,2\pi),\\
u(t,0)=u(t,2\pi)&\mbox{for }t\in\RR,\\
u_x(t,0)=u_x(t,2\pi)&\mbox{for }t\in\RR,\\
u(0,x)=u_0(x)&\mbox{for }x\in (0,2\pi).
\end{cases}
\end{equation}
Setting $L^2:=L^2(0,2\pi)$ for brevity and introducing the Sobolev space
\begin{equation*}
H^2_p:=\set{v\in H^2(0,2\pi)\ :\ v(0)=v(2\pi)\qtq{and}v_x(0)=v_x(2\pi)},
\end{equation*}
for each initial datum $u_0\in H^2_p$ there is a unique weak solution 
\begin{equation*}
u\in C(\RR,H^2_p)\cap C^1(\RR,L^2).
\end{equation*}
Furthermore, $u$ has a Fourier series representation
\begin{equation}\label{32}
u(t,x)=\sum_{k\in\ZZ}c_ke^{i(k^2t+kx)}
\end{equation}
where the $c_k$'s are the Fourier coefficients of $u_0$:
\begin{equation*}
u_0(x)=\sum_{k\in\ZZ}c_ke^{ikx}.
\end{equation*}
In particular, the $c_k$ satisfy Parseval's identity 
\begin{equation*}
\sum_{k\in\ZZ}\abs{c_k}^2=\frac{1}{2\pi}\norm{u_0}_{L^2}^2.
\end{equation*}
Using \eqref{32} we extend the solutions to $\RR^2$ by $2\pi$-periodicity in $x$ and $t$.

First we ask whether the observability of the solutions on a fixed line segment of $\RR^2$ allows us to identify the unknown initial datum. 

The case of vertical segments is easy: since the exponential functions $e^{ikx}$ form an orthogonal basis in $L^2(I)$ on every interval $I$ of length $2\pi$, we infer from the formula
\begin{equation*}
u(t_1,x)=\sum_{k\in\ZZ}\left(c_ke^{ik^2t_1}\right)e^{ikx}
\end{equation*}
that the knowledge of $u$ on a segment $\set{t_1}\times I$ determines $u_0$ if and only if $\abs{I}\ge 2\pi$.
Moreover, in the latter case we also have the quantitative relation 
\begin{equation*}
\int_I\abs{u(t_1,x)}\,\mbox{d}x\asymp \sum_{k\in\ZZ}\abs{c_k}^2.
\end{equation*}

The case of horizontal segments (pointwise observability) is different: we infer from the equality
\begin{equation*}
u(t,x_1)=\sum_{k\in\ZZ}\left(c_ke^{ikx_1}\right)e^{ik^2t}
=c_0+\sum_{k=1}^{\infty}e^{ik^2t}\left(c_ke^{ikx_1}+c_{-k}e^{-ikx_1}\right)
\end{equation*}
that the knowledge of $u$ even on the line on a segment $\RR\times \set{x_1}$ does not determine $u_0$.
For example, if $u_0(x)=e^{-ix_1}e^{ix}-e^{ix_1}e^{-ix}$, that is
 $c_1=e^{-ix_1}$, $c_{-1}=-e^{ix_1}$ and $c_k=0$ for all other $k$'s, then $u(t,x_1)=0$ for all $t\in\RR$,
although $u(t,x)$ is not the zero solution.

The situation is much better for most \emph{oblique} segments:
\medskip 

\begin{theorem}\label{t31}
Fix $(t_1,x_1)\in\RR^2$, $a\in\RR$ and $T>0$ arbitrarily, and consider the solutions of \eqref{31}. 
\begin{enumerate}[\upshape (i)]
\item The direct inequality
\begin{equation*}
\int_0^T\abs{u(t_1+t,x_1-at)}^2\,\mbox{d}t\ll \sum_{k\in\ZZ}\abs{c_k}^2
\end{equation*}
always holds.
\item If $a\notin\ZZ$, then the inverse inequality
\begin{equation}\label{33}
\sum_{k\in\ZZ}\abs{c_k}^2\ll \int_0^T\abs{u(t_1+t,x_1-at)}^2\,\mbox{d}t 
\end{equation}
also holds.
\item If $a\in\ZZ$, then 
\begin{equation}\label{34}
\int_0^T\abs{u(t_1+t,x_1-at)}^2\,\mbox{d}t\asymp \sum_{k\in\ZZ}^{\infty}\abs{d_k+d_{a-k}}^2,
\end{equation}
where we use the notations
\begin{equation*}
d_k:=c_ke^{i(k^2t_1+kx_1)},\quad k\in\ZZ.
\end{equation*}
In particular, then  there exist non-trivial solutions satisfying
\begin{equation}\label{35}
u(t_1+t,x_1-at)=0\qtq{for all}t\in\RR,
\end{equation} 
and therefore the inverse inequality \eqref{33} fails.
\end{enumerate}
\end{theorem}
Changing $u(t,x)$ to $v(t,x):=u(-t,x)$ we see that analogous results hold if we change the equation in \eqref{31} to $u_t-iu_{xx}=0$.

\begin{proof}
(i) For any fixed $a\in\RR$ a straightforward computation shows that 
\begin{equation}\label{36}
u(t_1+t,x_1-at) 
=\sum_{k\in\ZZ}c_ke^{i(k^2(t_1+t)+k(x_1-at))}=\sum_{k\in\ZZ}d_ke^{i(k^2-ak)t}.
\end{equation}
Since $\Lambda:=\set{k^2-ak\ :\ k\in\ZZ}$ is the union of 
\begin{equation*}
\set{k^2-ak\ :\ k\in\ZZ,\quad k\ge a/2} 
\qtq{and} 
\set{k^2-ak\ :\ k\in\ZZ,\quad k< a/2},
\end{equation*}
it suffices to show that latter two sets are uniformly discrete. 
(In view of Theorem A (i) this will also show that the restrictions of the solutions for segments are well defined.)
This follows from the following inequalities: if $k\ge a/2$
\begin{equation*}
\bigl((k+1)^2-a(k+1)\bigr)-(k^2-ak)=2k+1-a\ge 1
\end{equation*}
while if $k<a/2$,
\begin{equation*}
(k^2-ak)-\bigl((k-1)^2-a(k-1)\bigr)=2k+a-1\le -1.
\end{equation*}
\medskip 

(ii) If $a\notin\ZZ$, then the set $\set{k^2-ak\ :\ k\in\ZZ}$ itself is uniformly discrete. 
Indeed, if $k$ and $m$ are different integers, then
\begin{equation*}
|(k^2-ak)-(m^2-am)|=|k-m||k+m-a|\geq d(a,\ZZ):=\max(a-[a],[a]+1-a)
\end{equation*} 
where $[a]$ is the integer part of $a$.
If, for some  positive integer $N$, $k\not=m$ and $k,m\notin\set{-N,\ldots,N}$, then, using again the identity 
\begin{equation*}
(k^2-ak)-(m^2-am)=(k-m)(k+m-a),
\end{equation*}
we have 
\begin{equation*}
\abs{(k^2-ak)-(m^2-am)}\ge 
\begin{cases}
 2N-a&\text{if }k>m\ge N,\\
 2N+a&\text{if }k<m\le-N,\\
 2Nd(a,\ZZ)&\text{if }km<0.
\end{cases}
\end{equation*} 
It follows that 
\begin{equation*}
\gamma(\Lambda\setminus\set{-N,\ldots,N})\ge Nd(a,\ZZ)
\end{equation*}
for all integers $N>\abs{a}$.
Letting $N\to\infty$ and applying Theorem B we get the inverse inequality \eqref{33}. 
\medskip 

(iii) If $a\in\ZZ$, then we may rewrite \eqref{36} in the form 
\begin{equation*}
u(t_1+t,x_1-at) 
=d_{a/2}e^{-i(a^2/4)t}+\sum_{k\in\ZZ, k>a/2}(d_k+d_{a-k})e^{i(k^2-ak)t}
\end{equation*}
with the convention $d_{a/2}:=0$ if $a$ is an odd integer.

Since the set 
\begin{equation*}
\set{k^2-ak\ :\ k\in\ZZ,\quad k\ge a/2} 
\end{equation*}
is uniformly separated, and 
\begin{equation*}
\abs{(k^2-ak)-(m^2-am)}\ge 2N-a
\end{equation*}
whenever $k>m\ge N$, applying Theorem B we get \eqref{34} 
because
\begin{equation*}
\abs{d_{a/2}}^2+\sum_{k\in\ZZ,\ k>a/2}^{\infty}\abs{d_k+d_{a-k}}^2
\asymp \sum_{k\in\ZZ}^{\infty}\abs{d_k+d_{a-k}}^2.
\end{equation*}

It follows from \eqref{34} that all solutions satisfying $d_{a/2}=0$ and $d_k+d_{a-k}=0$ for all $k\in\ZZ$ satisfy the equality \eqref{35}.
If at least one of these coefficients is different from zero, then the right side of \eqref{33} vanishes, while the left side is positive.

A concrete nonzero function satisfying \eqref{35} may be given as follows. 
We choose an integer $k\ne a/2$ and then two nonzero numbers $c_k, c_{a-k}$ satisfying the equality 
\begin{equation*}
c_ke^{i(k^2t_1+kx_1)}+c_{a-k}e^{i((a-k)^2t_1+(a-k)x_1)}=0.
\end{equation*}
Then the function 
\begin{equation*}
u(t,x):=c_ke^{i(k^2t+kx)}+c_{a-k}e^{i((a-k)^2t+(a-k)x)}
\end{equation*}
has the required properties.
\end{proof}

Next we investigate what happens if we observe the solutions on two or more segments.
In view of Theorem \ref{t31} (i) we only investigate the validity of the inverse inequalities.

\begin{theorem}\label{t32}
Fix $T>0$ arbitrarily, and consider the solutions of \eqref{31}. 
\begin{enumerate}[\upshape (i)]
\item Let $(t_1,x_1),(t_2,x_2)\in\RR^2$,  and $a_1,a_2$ two different integers.
If
\begin{equation*}
u(t_1+t,x_1-a_1t)=u(t_2+t,x_2-a_2t)=0\qtq{for all}t\in(0,T),
\end{equation*}
then $u$ is the trivial solution, i.e., $u_0=0$.
Nevertheless, the inverse inequality
\begin{equation*}
\sum_{k\in\ZZ}\abs{c_k}^2\ll \int_0^T\abs{u(t_1+t,x_1-a_1t)}^2+\abs{u(t_2+t,x_2-a_2t)}^2\,\mbox{d}t 
\end{equation*}
fails.

\item The inverse inequality
\begin{equation*}
\sum_{k\in\ZZ}\abs{c_k}^2\ll \sum_{i=1}^m\int_0^T\abs{u(t_1+t,x_1-a_it)}^2\,\mbox{d}t 
\end{equation*}
also fails for any $(t_1,x_1)\in\RR^2$  and any finite number of integers $a_1,\ldots, a_m$.
\end{enumerate}
\end{theorem}

\begin{proof}
(i) Writing $d_k=c_ke^{i(k^2t_1+kx_1)}$ again, the equality $u(t_1+t,x_1-a_1t)=0$ for $t\in(0,T)$ together with \eqref{34} imply that $d_k=-d_{a_1-k}$
for every $k$. 
It follows that $|c_k|=|c_{a_1-k}|$ for every $k$.

Similarly,  $u(t_2+t,x_2-a_2t)=0$ for $t\in(0,T)$ implies $|c_k|=|c_{a_2-k}|$ for every $k$.
But then $|c_k|=|c_{a_2-k}|=|c_{a_1-(a_1-a_2+k)}|=|c_{a_1-a_2+k}|$, that is $|c_k|$
is $(a_1-a_2)$-periodic. As $|c_k|$ is square-summable, this can only happen if $c_k=0$ for every $k$.

For the second part, we construct a sequence $(u_{0,n})\subset L^2(0,2\pi)$ of initial data such that the corresponding solutions $u_n$ satisfy the relation
\begin{equation}\label{37}
\frac{1}{\norm{u_{0,n}}_2^2}\sum_{i=1}^2\int_0^T\abs{u_n(t_i+t,x_i-a_it)}^2\,\mbox{d}t \to 0.
\end{equation}

By the previous theorem the solutions satisfy the relation
\begin{equation*}
\sum_{i=1}^2\int_0^T\abs{u(t_i+t,x_i-a_it)}^2\,\mbox{d}t
\asymp \sum_{k\in\ZZ}^{\infty}\abs{d_k+d_{a_1-k}}^2+\abs{e_k+e_{a_2-k}}^2
\end{equation*}
with the notations
\begin{equation*}
d_k:=c_ke^{i(k^2t_1+kx_1)}
\qtq{and}
e_k:=c_ke^{i(k^2t_2+kx_2)}.
\end{equation*}
Assuming by symmetry that $a_2>a_1$, and setting  $p=a_2-a_1$ the relation may be rewritten in the form
\begin{equation*}
\sum_{i=1}^2\int_0^T\abs{u(t_i+t,x_i-a_it)}^2\,\mbox{d}t
\asymp \sum_{k\in\ZZ}^{\infty}\abs{d_k+d_{a_1-k}}^2+\abs{\omega_kd_k+d_{a_1-k+p}}^2
\end{equation*}
with suitable unimodular complex numbers $\omega_k$.

Fix an integer $q>\frac{a_1}{2}$.
For any fixed positive integer $n$ we define consecutively the following numbers $d_k$:
\begin{align*}
&d_q:=1,&&d_{a_1-q}:=- d_q,\\
&d_{q+p}:=-\omega_{a_1-q} d_{a_1-q},&&d_{a_1-q-p}:=-d_{q+p},\\
&d_{q+2p}:=-\omega_{a_1-q-p} d_{a_1-q-p},&&d_{a_1-q-2p}:=-d_{q+2p},\\
&\cdots\\
&d_{q+np}:=-\omega_{a_1-q-(n-1)p} d_{a_1-q-(n-1)p},&&d_{a_1-q-np}:=-d_{q+np}.
\end{align*}
Setting $d_k:=$ for all other indices, we obtain a trigonometric polynomial
\begin{equation*}
u_{0,n}(x)=\sum_{k\in\ZZ}c_ke^{ikx}:=\sum_{k\in\ZZ}d_ke^{-i(k^2t_1+kx_1)}e^{ikx}
\end{equation*}
satisfying 
\begin{equation*}
\norm{u_{0,n}}_2^2=2\pi\sum_{k\in\ZZ}\abs{c_k}^2=2n+2
\end{equation*}
and
\begin{equation*}
\sum_{k\in\ZZ}^{\infty}\abs{d_k+d_{a_1-k}}^2+\abs{\omega_kd_k+d_{a_1-k+p}}^2=2.
\end{equation*}
This proves \eqref{37}.

\medskip 

(ii) Note that in this part, $(x_1,t_1)$ is the same for each $a_1,\ldots,a_m$. 
We will take advantage of this to construct the sequence $d_k$.

Now we are looking for a sequence $(u_{0,n})\subset L^2(0,2\pi)$ of initial data such that the corresponding solutions $u_n$ satisfy the relation
\begin{equation*}
\frac{1}{\norm{u_{0,n}}_2^2}\sum_{i=1}^m\int_0^T\abs{u_n(t_1+t,x_1-a_it)}^2\,\mbox{d}t \to 0.
\end{equation*}
For any fixed positive integer $n$ we consider the  numbers
\begin{equation*}
d_k:=
\begin{cases}
\sgn k&\qtq{if}\abs{k}\le n,\\
0&\qtq{if}\abs{k}>n,
\end{cases}
\end{equation*}
and we define
\begin{equation*}
u_n(x):=\sum_{k=-n}^nc_ke^{ikx}
\qtq{with}
c_k:=d_ke^{-i(k^2t_1+kx_1)}.
\end{equation*} 
Then 
\begin{equation*}
\norm{u_{0,n}}^2_2=2\pi\sum_{k\in\ZZ}\abs{c_k}^2=2\pi\sum_{k\in\ZZ}\abs{d_k}^2=4n\pi.
\end{equation*}
On the other hand, using \eqref{34},
\begin{equation*}
\sum_{i=1}^m\int_0^T\abs{u(t_1+t,x_1-a_it)}^2\,\mbox{d}t\asymp\sum_{i=1}^m\sum_{k\in\ZZ}^{\infty}\abs{d_k+d_{a_i-k}}^2.
\end{equation*}
Therefore we will reach a contradiction if we bound
$\displaystyle \sum_{k\in\ZZ}^{\infty}\abs{d_k+d_{a_i-k}}^2$ independently of $n$.

Fix $i$ arbitrarily and write $a:=a_i$ for brevity.
The sequence $(d_k+d_{a-k})$ takes only the values $-2,-1,0,1,2$. 
It suffices to show that the number of $k$'s for which $d_k+d_{a-k}\ne 0$ is $\ll 1+\abs{a}$.
By the symmetry of the sequence $(d_k)$ it suffices to consider the values $1$ and $2$. 

We have 
\begin{equation*}
d_k+d_{a-k}=2\Longleftrightarrow 1\le k\le n\qtq{and}1\le a-k\le n\Longrightarrow 1\le k\le 1+a,
\end{equation*}
so that we have either no such $k$ if $a<0$ or at most $1+a$ such indices $k$ if $a>0$.

Next, we have $d_k+d_{a-k}=1$ in the following three cases:
\begin{align*}
&1\le k\le n\qtq{and}a-k\ge n+1\Longrightarrow 1\le k\le a-1-n;\\
&1\le k\le n\qtq{and}a-k=0;\\
&1\le k\le n\qtq{and}a-k\le -n-1\Longrightarrow a+1+n\le k\le n,
\end{align*}
and in three other symmetric cases by exchanging $k$ and $a-k$.

Since the first two cases above may only occur for $a>0$, while the third case only for $a<0$, at most 
\begin{equation*}
\max\set{(a-1-n)+1,-a}\le \abs{a}
\end{equation*} 
indices $k$ satisfy one of them.
We have the same upper bound for the three symmetric cases, so that there are at most $2\abs{a}$ indices $k$ for which $d_k+d_{a-k}=1$.
\end{proof}

We may also consider other boundary conditions.
Let us consider for example the Dirichlet condition:
\begin{equation}\label{38}
\begin{cases}
u_t+iu_{xx}=0&\mbox{in }\RR\times(0,\pi),\\
u(t,0)=u(t,\pi)=0&\mbox{for }t\in\RR,\\
u(0,x)=u_0(x)&\mbox{for }x\in (0,\pi).
\end{cases}
\end{equation}
The problem is well posed for every $u_0\in H_0^2(0,\pi)$.
Let us observe that extending an arbitrary solution of \eqref{38} to a $2\pi$-periodic odd function in the $x$ variable we obtain a solution of \eqref{31}.
Therefore Theorem \ref{t31} (i), (ii), (iii) and Theorem \ref{t32} (i) remain valid for the solutions of \eqref{38}.

The remaining parts were based on the construction of special solutions, so we need some additional arguments.
We have the following

\begin{proposition}\label{p34}
Fix $T>0$ arbitrarily, and consider the solutions of \eqref{38}. 

\begin{enumerate}[\upshape (i)]
\item  For any given  $(t_1,x_1)\in\RR^2$ and $a\in\ZZ$  there exist non-trivial solutions of \eqref{38} satisfying
\begin{equation*}
u(t_1+t,x_1-at)=0\qtq{for all}t\in\RR.
\end{equation*} 
\item The inverse inequality 
\begin{equation*}
\sum_{k\in\ZZ}\abs{c_k}^2\ll \sum_{i=1}^m\int_0^T\abs{u(t,-a_it)}^2\,\mbox{d}t 
\end{equation*}
fails for any $T>0$ and for any finite number of integers 
$a_1,\ldots, a_m$.
\end{enumerate}
\end{proposition}

\begin{proof}
The solutions of \eqref{38}  are given by the series 
\begin{equation*}
u(t,x)=\sum_{k\in\ZZ}c_ke^{i(k^2t+kx)}
\end{equation*}
with suitable square summable complex coefficients $c_k$ satisfying the relations $c_k+c_{-k}=0$.
\medskip 

(i) We choose an integer $k$ for which the four numbers $k, a-k, -k, k-a$ are different that  and then two nonzero numbers $c_k, c_{a-k}$ satisfying the equality 
\begin{equation*}
c_ke^{i(k^2t_1+kx_1)}+c_{a-k}e^{i((a-k)^2t_1+(a-k)x_1)}=0.
\end{equation*}
Then the function 
\begin{align*}
u(t,x):&=c_k\sin(k^2t+kx)+c_{a-k}\sin((a-k)^2t+(a-k)x)\\
&=\left(\frac{c_k}{2i}e^{i(k^2t+kx)}+\frac{c_{a-k}}{2i}e^{i((a-k)^2t+(a-k)x)}\right)-\left(\frac{c_k}{2i}e^{i(k^2t-kx)}+\frac{c_{a-k}}{2i}e^{i((a-k)^2t-(a-k)x)}\right)
\end{align*}
has the required properties by the same arguments as in the proof of Theorem \ref{t31} (iv).
\medskip 

(ii) Since $(t_1,x_1)=(0,0)$ and therefore $d_k=c_k$ for all $k$, the sequences constructed in the proof of Theorem \ref{t32} (ii) define solutions of not only \eqref{31}, but also of \eqref{38}.
\end{proof}

\section{Beam equation}\label{s4}

We consider the one-dimensional linear beam equation with periodic boundary conditions:\begin{equation}\label{41}
\begin{cases}
u_{tt}+u_{xxxx}=0&\mbox{in }\RR\times(0,2\pi),\\
u(t,0)=u(t,2\pi)&\mbox{for }t\in\RR,\\
u_x(t,0)=u_x(t,2\pi)&\mbox{for }t\in\RR,\\
u(0,x)=u_0(x)&\mbox{for }x\in (0,2\pi),\\
u_t(0,x)=u_1(x)&\mbox{for }x\in (0,2\pi).
\end{cases}
\end{equation}
For any given initial data $u_0\in H^2_p$ and $u_1\in L^2$ there is a unique weak solution 
\begin{equation*}
u\in  C(\RR,H^2_p)\cap C^1(\RR,L^2).
\end{equation*}
Furthermore, $u$ has a Fourier series representation
\begin{equation}\label{42}
u(t,x)=c_0^++c_0^-t+\sum_{k\in\ZZ^*}\left(c_k^+e^{i(k^2t+kx)}+c_k^-e^{i(-k^2t+kx)}\right)
\end{equation}
with suitable square summable complex coefficients $c_k^+, c_k^-$ satisfying the relations 
\begin{equation*}
\sum_{k\in\ZZ}(1+k^4)(\abs{c_k^+}^2+\abs{c_k^-}^2)\asymp \norm{u_0}_{H^2_p}^2+\norm{u_1}_{L^2}^2.
\end{equation*}
Using \eqref{42} we extend the solutions to $\RR^2$ by $2\pi$-periodicity in $x$.

\begin{remark}\label{r41}
Observe that \eqref{42} is no longer a trigonometric series if $c_0^-\ne 0$.
However, the results and  proofs of this section remain valid in the general case by Remark \ref{r23}.
\end{remark}

First we consider the observation of the solutions on vertical line segments.
(Analogous theorems have been proved in \cite{Szijarto-Hegedus-2012} for the Klein--Gordon equations by a different approach.)

\begin{theorem}\label{t42}
Fix two distinct nonzero real numbers $t_1, t_2$, a number $T>0$, and consider the solutions of  \eqref{41}.

\begin{enumerate}[\upshape (i)]
\item The direct inequality
\begin{equation}\label{43}
\int_0^T|u(t_1,x)|^2+|u(t_2,x)|^2\,\mbox{d}x
\ll
\sum_{k\in\ZZ}(|c_k^+|^2+|c_k^-|^2)
\end{equation}
and the weakened inverse inequality
\begin{equation}\label{44}
|c_0^+|^2+|c_0^-|^2+\sum_{k\in\ZZ^*}\sin^2 k^2(t_1-t_2)(|c_k^+|^2+|c_k^-|^2)
\ll
\int_0^T|u(t_1,x)|^2+|u(t_2,x)|^2\,\mbox{d}x
\end{equation}
always hold.

\item If $(t_2-t_1)/\pi$ is irrational, then the right hand side of \eqref{44} does not vanish for any non-trivial solution.

\item If $(t_2-t_1)/\pi$ is rational, there exist non-trivial solutions for which the right hand side of \eqref{44} vanishes.

\item The inverse inequality
\begin{equation}\label{45}
\sum_{k\in\ZZ}\left(\abs{c_k^+}^2+\abs{c_k^-}^2\right)
\ll\int_0^T\abs{u(t_1,x)}^2+\abs{u(t_2,x)}^2\,\mbox{d}x 
\end{equation}
always fails.
\end{enumerate}
\end{theorem}

\begin{proof}
(i) Since
\begin{equation*}
u(t_j,x)=c_0^++c_0^-t_j+\sum_{k\in\ZZ^*}\left(c_k^+e^{ik^2t_j}+c_k^-e^{-ik^2t_j}\right)e^{ikx},
\end{equation*}
for $j=1,2$, applying Theorem B we get the relations
\begin{equation}\label{46}
\abs{c_0^++c_0^-t_j}^2+\sum_{k\in\ZZ^*}\abs{c_k^+e^{ik^2t_j}+c_k^-e^{-ik^2t_j}}^2
\asymp \int_0^T|u(t_j,x)|^2\,\mbox{d}x.
\end{equation}
They imply \eqref{43} by using the elementary inequality $\abs{a+b}^2\le2\abs{a}^2+2\abs{b}^2$.

The relations \eqref{44} follows by adding \eqref{46} for $j=1,2$, and using for each $k\in\ZZ^*$ the following estimates with $a=k^2t_1$ and $b=k^2t_2$:
\begin{align*}
\abs{xe^{ia}+ye^{-ia}}^2+\abs{xe^{ib}+ye^{-ib}}^2
&=\abs{xe^{2ia}+y}^2+\abs{xe^{2ib}+y}^2\\
&=2(\abs{x}^2+\abs{y}^2)+2\Re\left(x\overline{y}\left(e^{2ia}+e^{2ib}\right)\right)\\
&\ge (\abs{x}^2+\abs{y}^2)\left(2-\abs{e^{2ia}+e^{2ib}}\right)\\
&=(\abs{x}^2+\abs{y}^2)\frac{4-\abs{e^{2ia}+e^{2ib}}^2}{2+\abs{e^{2ia}+e^{2ib}}}\\
&=(\abs{x}^2+\abs{y}^2)\frac{4\sin^2(a-b)}{2+\abs{e^{2ia}+e^{2ib}}}\\
&\ge (\abs{x}^2+\abs{y}^2)\sin^2(a-b).
\end{align*}

(ii) If $(t_2-t_1)/\pi$ is irrational and the right side of \eqref{44} vanishes for some solution $u$, then we infer from \eqref{44} that all coefficients $c_k^{\pm}$ are equal to zero because $\sin^2 k^2(t_1-t_2)\ne 0$ for all $k\in\ZZ^*$, so that $u$ is the trivial solution.

(iii) If $(t_2-t_1)/\pi$ is rational, there exists a nonzero integer $k$ such that $k^2(t_1-t_2)$ is a multiple of $2\pi$. 
Then the formula
\begin{equation*}
u(t,x)=\left(e^{-ik^2t_1}e^{ik^2t}-e^{ik^2t_1}e^{-ik^2t}\right)e^{ikx}
\end{equation*}
defines a non-trivial solution of \eqref{41} such that $u(t_1,x)=u(t_2,x)=0$ for all $x\in\RR$.

(iv) It follows from \eqref{46} that the inverse inequality \eqref{45} holds if and only if the matrices 
\begin{equation*}
A_k:=
\begin{pmatrix}
e^{ik^2t_1}&e^{-ik^2t_1}\\
e^{ik^2t_2}&e^{-ik^2t_2}
\end{pmatrix}
\end{equation*}
are invertible, and the norms of their inverses are bounded by some uniform constant.

If $(t_2-t_1)/\pi$ is rational, then not all matrices $A_k$ are invertible by (iii).
Otherwise, by the irrationality there exists a sequence $(k_j)$ of positive integers such that $k_j^2(t_1-t_2)\to 0\mod 2\pi$, and then the above norms tend to $\infty$ as $j\to\infty$.
\end{proof}

Now we turn to the case of oblique segments.
Given a real number $a$, if $u$ is a solution of \eqref{41}, then a straightforward computation shows that 
\begin{equation}\label{47}
u(t_0+t,x_0-at) 
=d_0^++d_0^-t+\sum_{k\in\ZZ^*}\left(d_k^+e^{i(k^2-ak)t}+d_k^-e^{i(-k^2-ak)t}\right),
\end{equation}
where we use the notations
\begin{equation*}
d_0^+:=c_0^++c_0^-t_0,\quad 
d_0^-:=c_0^-
\end{equation*}
and
\begin{equation*}
d_k^+:=c_k^+e^{i(k^2t_0+kx_0)},\quad 
d_k^-:=c_k^-e^{i(-k^2t_0+kx_0)}
\qtq{for}k\in\ZZ^*.
\end{equation*}
Observe that
\begin{equation*}
\sum_{k\in\ZZ}\left(\abs{d_k^+}^2+\abs{d_k^-}^2\right)
\asymp \sum_{k\in\ZZ}\left(\abs{c_k^+}^2+\abs{c_k^-}^2\right).
\end{equation*}

In order to state our results we introduce the circle $S_a\subset\RR^2$  centered in $(a/2,-a/2)$ and passing through the origin.
Its cartesian equation is
\begin{equation*}
\left(x-\frac{a}{2}\right)^2+\left(y+\frac{a}{2}\right)^2=\frac{a^2}{2}
\qtq{or equivalently}
x^2-ax+y^2+ay=0.
\end{equation*} 
Furthermore, we introduce the set
$A_a=S_a\cap \ZZ^2\setminus\set{(0,0)}$.

\begin{remarks}\label{r43}\mbox{}
\begin{enumerate}[\upshape (i)]
\item Since the distance between distinct elements of $A_a$ is at least one, $A_a$ cannot have more elements than the perimeter of the circle $S_a$: $\abs{A_a}\le\sqrt{2}\pi a$.

\item If $a$ is irrational, then $A_a$ is empty.
Indeed, if $(k,m)\in A_a$, then $a=\frac{k^2+m^2}{k-m}\in\QQ$.

\item If $a$ is not an integer, then no element of $A_a$ has any zero coordinate, and hence $A_a=(\ZZ^*)^2\cap S_a$.
Indeed, if $(k,0)\in A_a$, then $a=k\in\ZZ$ from the above equation of $S_a$.
\end{enumerate}
\end{remarks}

\begin{theorem}\label{t44}
Fix $(t_1,x_1)\in\RR^2$, $a\in\RR$ and $T>0$ arbitrarily, and consider the solutions of \eqref{41}. 

\begin{enumerate}[\upshape (i)]
\item The direct inequality
\begin{equation*}
\int_0^T\abs{u(t_1+t,x_1-at)}^2\,\mbox{d}t\ll \sum_{k\in\ZZ}\left(\abs{c_k^+}^2+\abs{c_k^-}^2\right)
\end{equation*}
always holds.

\item If $a\ne 0$ and $A_a=\varnothing$, then the inverse inequality
\begin{equation*} 
\sum_{k\in\ZZ}\left(\abs{c_k^+}^2+\abs{c_k^-}^2\right)\ll \int_0^T\abs{u(t_1+t,x_1-at)}^2\,\mbox{d}t 
\end{equation*}
also holds.
In particular, the inverse inequality holds whenever $a$ is irrational.

\item If $a\in\ZZ$ or if $A_a\ne\varnothing$, then there exist non-trivial solutions satisfying
\begin{equation*}
u(t_1+t,x_1-at)=0\qtq{for all}t\in\RR,
\end{equation*} 
so that the inverse inequality in (ii) fails.
\end{enumerate}
\end{theorem}

\begin{remark}\label{r45}
Similarly to the Schr\"odinger equation, analogous results may be obtained for other boundary conditions; the details are left to the reader.
\end{remark}

\begin{proof}[Proof of Theorem \ref{t44}]
(i) The proof of Theorem \ref{t31} (i) shows that the exponents in \eqref{47} form a finite union of uniformly separated sets.
Hence the direct inequality holds by Theorem A and the remarks following Theorems A and B.
\medskip 

(ii) Since $(a,-a)\in A_a$ for all nonzero integers, $a\notin\ZZ$ by our assumptions, and therefore both sets
\begin{equation*}
\set{k^2-ak\ :\ k\in\ZZ}
\qtq{and}
\set{-k^2-ak\ :\ k\in\ZZ}
\end{equation*}
are uniformly discrete by the proof of Theorem \ref{t31} (ii).
In view of \eqref{47} we have to show that their union is also uniformly discrete.

This amounts to show that 
\begin{equation*}
\inf\set{\abs{(k^2-ak)-(-m^2-am)}\ :\ (k,m)\in\ZZ^2\setminus\set{(0,0)}}>0.
\end{equation*}
Since 
\begin{equation*}
(k^2-ak)-(-m^2-am)=\left(k-\frac{a}{2}\right)^2+\left(m+\frac{a}{2}\right)^2-\frac{a^2}{2},
\end{equation*}
this means that the circle $S_a$ has a positive distance from the set $\ZZ^2\setminus\set{(0,0)}$.
Since the latter set is discrete, this is satisfied by our assumption $A_a=\varnothing$.
\medskip 

(iii) If $a\in\ZZ$, then the function given in the proof of Theorem \ref{t31} (iii) also solves \eqref{41}.
Otherwise choose $(k,m)\in S_a\cap\ZZ^2\setminus\set{(0,0)}$, and set 
\begin{equation*}
u(t,x):=e^{i(-m^2t_1+mx_1)}e^{i(k^2t+kx)}-e^{i(k^2t_1+kx_1)}e^{i(-m^2t+mx)}.
\end{equation*}
Then $d_k^++d_m^-=0$ and therefore 
\begin{equation*}
u(t_0+t,x_0-at) 
=(d_k^++d_m^-)e^{i(k^2-ak)t}
=0
\end{equation*}
for all $t\in\RR$.
\end{proof}

Now let us investigate the inverse inequality 
\begin{equation}\label{48}
\sum_{k\in\ZZ}\left(\abs{c_k^+}^2+\abs{c_k^-}^2\right)\ll \int_0^T\abs{u(t_1+t,x_1-a_1t)}^2+\abs{u(t_2+t,x_2-a_2t)}^2\,\mbox{d}t
\end{equation}
when we observe the solutions on two segments.

We start with some simple observations.
We write $A_j$ instead of $A_{a_j}$ for brevity, and we denote by $A_j^+, A_j^-$ its projection on the first and second coordinate axis, respectively.

\begin{remarks}\label{r46}\mbox{}
\begin{enumerate}[\upshape (i)]
\item If $a_1\ne 0$ and $A_1=\varnothing$, then \eqref{48} holds by the preceding theorem. 
The same conclusion holds by symmetry if $a_2\ne 0$ and $A_2=\varnothing$.
\item The proof of Theorem \ref{t44} (ii) shows that 
\begin{equation}\label{49}
\sum_{k\in\ZZ\setminus A_j^+}\abs{d_k^+}^2+\sum_{m\in\ZZ\setminus A_j^-}\abs{d_m^-}^2
+\sum_{(k,m)\in A_j}\abs{d_k^++d_m^-}^2
\asymp \int_0^T\abs{u(t_j+t,x_j-a_jt)}^2\,\mbox{d}t
\end{equation}
for $j=1,2$.
If $A_1^+\cap A_2^+=A_1^-\cap A_2^-=\varnothing$, then adding these estimates for $j=1,2$ the inequality \eqref{48} follows.

\item On the other hand, if $a_1=a_2$ and $A_1\ne\varnothing$, then \eqref{48} fails.
Indeed, if $(k,m)\in A_1$ and $c\in\RR$, then changing $d_k^+$ and $+d_m^-$ to $d_k^++c$ and $+d_m^--c$ the right side of \eqref{49} remains unchanged.
\end{enumerate}
\end{remarks}

\begin{lemma}\label{l47}
If $a_1, a_2$ are different nonzero integers, then

\begin{enumerate}[\upshape (i)]
\item $A_1\cap A_2=\varnothing$;

\item if $(k,m),(k',m')$ are two different points in $A_1$ or $A_2$, then $k\ne k'$ and $m\ne m'$.
\end{enumerate}
\end{lemma}

\begin{proof}
(i) Since all circles $S_a$ have the same tangent line in the origin, $S_1\cap S_2=\set{(0,0)}$, and hence $A_1\cap A_2=\varnothing$.
\medskip

(ii) If for example $(k,m),(k',m)\in A_1$ with $k\ne k'$, then both $k$ and $k'$  solve the equation $x^2-a_1x+m^2+a_1m=0$, and hence $a_1=k+k'\in\ZZ$.
The other cases are similar.
\end{proof}

At this stage it is convenient to associate a graph $G(a_1,a_2)$ to a pair of distinct integers $a_1, a_2$. 
The vertices of this graph form the set $A_1\cup A_2$. 
Two vertices are adjacent if they have a common coordinate. 
The previous lemma then states that this graph is bipartite, namely a vertex in $A_1$ can only be adjacent to a vertex in $A_2$ and vice versa. 
A direct consequence of this is that a vertex has at most two neighbours.

A (simple) path is a sequence of distinct vertices $v_1,v_2,\ldots,v_n$ where $v_j,v_{j+1}$ are adjacent for every $j=1,\ldots,n-1$.
In particular, if the first coordinate is common in $v_j,v_{j+1}$ then $v_{j+1},v_{j+2}$ have the second coordinate in common, and vice versa.

A simple path has at most $|A_1|+|A_2|\leq \sqrt{2}\pi(a_1+a_2)$ elements.
Furthermore, every $v\in A_1\cup A_2$ belongs to a unique maximal simple path 
$v_{-\ell_1},\ldots, v_0,\ldots, v_{\ell_2}$ (see the figure).
This maximal path is called a \emph{cycle} if $v_{-\ell_1}$ and $v_{\ell_2}$ are adjacent, that is, they 
have a common component. 
Note that a cycle has necessarily an even number of points.

\begin{theorem}\label{t48}
Fix $(t_1,x_1),(t_2,x_2)\in \RR^2$, two different nonzero integers $a_1, a_2$ and $T>0$. 
The inverse inequality \eqref{48} fails if and only if $G(a_1,a_2)$ has a cycle.
\end{theorem}

\begin{proof}
Adding the relations \eqref{49} for $j=1,2$ we see that the right hand side of \eqref{48}
is
\begin{equation}
\asymp\sum_{k\in(\ZZ\setminus A_1^+)\cup (\ZZ\setminus A_2^+)}\abs{d_k^+}^2
+\sum_{m\in(\ZZ\setminus A_1^-)\cup (\ZZ\setminus A_2^-)}\abs{d_m^-}^2
+\sum_{(k,m)\in A_1\cup A_2}\abs{d_k^++d_m^-}^2.
\label{410}
\end{equation}

Since $A_1\cup A_2$ is finite, the difference between \eqref{410} and the
left hand side of \eqref{48} is a quadratic form in a finite number of variables.
Therefore  \eqref{48} is equivalent to the following uniqueness property: if the expression in \eqref{410} is zero, then all coefficients $d_k^\pm$ vanish.

Assume first that $G(a_1,a_2)$ has a cycle and write it as
\begin{equation*}
(k_1,m_1), (k_2,m_1), (k_2,m_2),\ldots, (k_n,m_{n-1}),(k_n,m_n), (k_1,m_n).
\end{equation*}
Up to exchanging $A_1$ and $A_2$ we may assume that $(k_1,m_1)\in A_1$, so that $(k_1,m_n)\in A_2$.
Note that $k_1,\ldots,k_{n-1}\in A_1^+\cup A_2^+$ and $m_1,\ldots,m_n\in A_1^-\cup A_2^-$.

Then setting $d_{k_i}^+=1$ and $d_{m_i}^-=-1$ for $i=1,\ldots,n$ and $d_k^+=d_m^-=0$ for all other indices $k,m$, the expression in \eqref{410} vanishes.

To prove the other direction, assume that $G(a_1,a_2)$ has no cycle, and consider an arbitrary maximal simple  path.
By symmetry between $a_1$ and $a_2$, we may assume that this path starts in $A_1$.
Depending on whether the first move is horizontal or vertical, and whether the path ends in $A_1$ or $A_2$, there are four possibilities:

\begin{itemize}
\item The first move is horizontal and the path ends in $A_2$, so that the path has the form
\begin{equation*}
(k_1,m_1),(k_2,m_1),(k_2,m_2),\ldots, (k_n,m_{n-1})
\end{equation*}
with $(k_1,m_1)\in A_1$ and $(k_n,m_{n-1})\in A_2$. 
Since the path is maximal, the only element of $A_2$ adjacent to $(k_1,m_1)$
is $(k_2,m_1)$; hence $k_1\notin A_2^+$. Similarly, $k_n\notin A_1^+$.

\item The first move is horizontal and the path ends in $A_1$: we have
\begin{equation*}
(k_1,m_1),(k_2,m_1),(k_2,m_2),\ldots, (k_n,m_{n-1}),(k_n,m_n)
\end{equation*}
with $(k_1,m_1)\in A_1$ and $(k_n,m_n)\in A_1$.
The maximality implies that $k_1\notin A_2^+$ and $m_n\notin A_2^-$.

\item The first move is vertical and the path ends in $A_2$: we have
\begin{equation*}
(k_1,m_1),(k_1,m_2),(k_2,m_2),\ldots, (k_{n-1},m_n)
\end{equation*}
with $(k_1,m_1)\in A_1$ and $(k_{n-1},m_n)\in A_2$. 
The maximality implies that $m_1\notin A_2^-$ and $m_n\notin A_1^-$

\item The first move is vertical and the path ends in  $A_1$: we have
\begin{equation*}
(k_1,m_1),(k_1,m_2),(k_2,m_2),\ldots, (k_{n-1},m_n),(k_n,m_n)
\end{equation*}
with $(k_1,m_1)\in A_1$ and $(k_n,m_n)\in A_1$.
The maximality implies that $m_1\notin A_2^-$ and $k_n\notin A_2^+$.
\end{itemize}

Let us consider the first case; the others are similar.
If the expression in \eqref{410} vanishes, then $d_k^\pm=0$ whenever $k\notin A_1^\pm$ or $k\notin A_2^\pm$, and
\begin{equation*}
\sum_{(k,m)\in A_1\cup A_2}\abs{d_k^++d_m^-}^2=0.
\end{equation*}
In particular, $d_{k_1}^+=0$, $d_{k_n}^+=0$ and
\begin{equation*}
|d_{k_1}^++d_{m_1}^-|^2+|d_{k_2}^++d_{m_1}^-|^2+\cdots
+|d_{k_n}^++d_{m_{n-1}}^-|^2=0,
\end{equation*}
whence $d_{k_j}^++d_{m_j}^-=d_{k_{j+1}}^++d_{m_j}^-=0$ for $j=1,\ldots,n-1$.
A direct induction then shows that $d_{k_j}^+=d_{m_j}^-=0$ for $j=1,\ldots,n-1$.

Finally, since every $(k,j)\in A_1\cup A_2$ belongs to a maximal simple  path, we conclude that 
$d_k^+=d_m^-=0$ for all $k,m$ as claimed.
\end{proof}

\begin{example}
The figure below shows the situation where $a_1=-\frac{13}{5}$
and $a_2=-\frac{10}{3}$.
There are two  maximal simple paths $KM$ and $K'M'$ having more than one element, and there are no cycles.
\end{example}

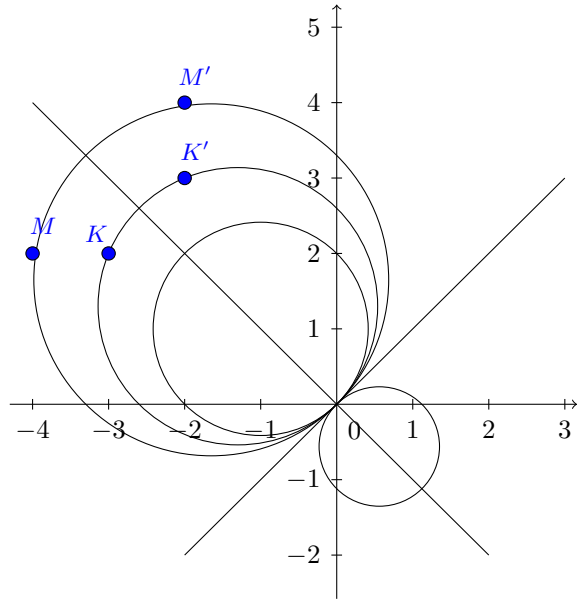
\begin{center}
\definecolor{qqqqff}{rgb}{0.,0.,1.}
\begin{tikzpicture}
\draw[->,color=black] (-4.3,0.) -- (3.16,0.);
\foreach \x in {-4,-3,-2,-1,1,2,3}
\draw[shift={(\x,0)},color=black] (0pt,2pt) -- (0pt,-2pt) node[below] {\footnotesize $\x$};
\draw[->,color=black] (0.,-2.58) -- (0.,5.3);
\foreach \y in {-2,-1,1,2,3,4,5}
\draw[shift={(0,\y)},color=black] (2pt,0pt) -- (-2pt,0pt) node[left] {\footnotesize $\y$};
\draw[color=black] (0pt,-10pt) node[right] {\footnotesize $0$};
\draw (-2,-2)-- (3,3);
\draw (-4,4)-- (2,-2);
\draw(-1.,1.) circle (1.4142135623730951cm);
\draw(0.56,-0.56) circle (0.7919595949289333cm);
\draw(-1.65,1.65) circle (2.3334523779156067cm);
\draw(-1.3,1.3) circle (1.8384776310850237cm);
\begin{scriptsize}
\draw [fill=qqqqff] (-3.,2.) circle (2.5pt);
\draw[color=qqqqff] (-3.16,2.26) node {$K$};
\draw [fill=qqqqff] (-2.,3.) circle (2.5pt);
\draw[color=qqqqff] (-1.86,3.36) node {$K'$};
\draw [fill=qqqqff] (-2.,4.) circle (2.5pt);
\draw[color=qqqqff] (-1.86,4.36) node {$M'$};
\draw [fill=qqqqff] (-4.,2.) circle (2.5pt);
\draw[color=qqqqff] (-3.86,2.36) node {$M$};
\end{scriptsize}
\end{tikzpicture}
\end{center}

We have no concrete example in which $G(a_1,a_2)$ has a cycle.
The following proposition indicates that if there exist such examples, they are rare.

\begin{proposition}\label{p49}\mbox{}
The graph $G(a_1,a_2)$ has no cycle in the following cases:
\begin{enumerate}[\upshape (i)]
\item $a_1$ and $a_2$ have opposite nonzero signs;
\item $a_1$ and $a_2$ have equal signs, and $a_1/a_2\ge 3/2$.
\end{enumerate}
\end{proposition}

\begin{proof}
(i) Without loss of generality, we may assume that $a\geq b>0$ and $a_1=-a$, $a_2=b$.

Let us start with the observation that, to belong to a cycle, a point $(k,m)\in S_{-a}$ 
must have two neighbours $(k,m'),(k',m)\in S_b$.

Set
\begin{equation*}
\tilde b=\left(\frac{1}{\sqrt{2}}-\frac{1}{2}\right)b=\frac{\sqrt{2}-1}{2}b,
\end{equation*}
and denote by $\tilde a_+$ and $\tilde a_-$  the positive and negative root of $x^2+ax+\tilde b^2-a\tilde b=0$,
that is
\begin{equation*}
\tilde a_+=\frac{-a+\sqrt{a^2+4a\tilde b-4\tilde b^2}}{2},
\quad \tilde a_-=\frac{-a-\sqrt{a^2-4a\tilde b-4\tilde b^2}}{2}.
\end{equation*}

Finally, let $P=(-\tilde b,\tilde a_-)$ (resp. $Q=(-\tilde a_-,\tilde b)$) be the point on
$S_{-a}$ such that the (vertical) (resp. horizontal)
line through $P$ and $(-\tilde b,-b/2)$ (resp. $Q$ and $(b/2,\tilde b)$) is tangent to $S_b$.

\begin{center}
\definecolor{qqqqff}{rgb}{0.,0.,1.}
\begin{tikzpicture}
\draw[->,color=black] (-4.3,0.) -- (3.16,0.);
\draw[->,color=black] (0.,-2.58) -- (0.,5.3);
\draw (-2,-2)-- (3,3);
\draw (-4,4)-- (2,-2);
\draw(1.,-1.) circle (1.4142135623730951cm);
\draw(-1.3,1.3) circle (1.8384776310850237cm);
\draw[-,color=black] (-0.41,-2) -- (-0.41,4.);
\draw[-,color=black] (-4,0.41) -- (3,0.41);
\begin{scriptsize}
\draw[color=black] (2.4,-2.1) node {$S_{b}$};
\draw[color=black] (-3.4,2.32) node {$S_{-a}$};
\draw[color=black] (-0.6,-0.25) node {$P$};
\draw[color=black] (0.25,0.6) node {$Q$};
\end{scriptsize}
\end{tikzpicture}\\
{\sc Figure:} Case (i)
\end{center}

Note that $P$ and $Q$ divide $S_{-a}$ into two arcs. Denote by $\Gamma^0$ the
one through $(0,0)$ and set
\begin{equation*}
\Gamma^0_+:=\Gamma^0\cap\{(x,y), x\geq 0,y\geq 0\}\qtq{and}
\Gamma^0_-:=\Gamma^0\cap\{(x,y), x\leq 0,y\leq 0\}.
\end{equation*} 

Now observe that a vertex of $G(-a,b)$ on $S_{-a}\setminus \Gamma^0$ has at most one neighbour, so that it cannot belong to a cycle.
Furthermore, a vertex of $G(-a,b)$ on $\Gamma^0_+$ has no neighbour on $\Gamma^0_-$.
Therefore, if there exists a cycle, then its
points on $S_{-a}$ should all belong to $\Gamma^0_+$, or should all belong to $\Gamma^0_-$.
Since the reflection of a cycle with respect to the anti-diagonal is also a cycle,
it remains to prove that there is no cycle all of whose points on $S_{-a}$ belong to $\Gamma^0_+$.

The geometric property of $\Gamma^0_+$ that we use is the following: consider the arc $\tilde \Gamma$
of $S_b$ joining $(0,0)$ to $(b/2,\tilde b)$. Then if we start at $(x,y)\in\Gamma^0_+$, draw a vertical line till we reach $\tilde \Gamma$ at some point $(x,y')$ and then draw an horizontal line $\ell$.
Then $\ell$ will intersect $S_{-a}$ at a point $(x',y')\in\Gamma^0_+$ with $0<x'<x$
(and at a second point $(x'',y')\notin\Gamma^0$). 

Now it is easy to see that $\Gamma^0_+$ contains no cycle.
Indeed, if $A_0=(k,l)\in G(-a,b)\cap \Gamma^0_+$ belonged to a cycle, then it would have  a neighbour of the form $A_1=(k,l')\in \tilde\Gamma$.
Then the second neighbour of $A_1$ would be of the form $(k',l')\in \Gamma^0_+$ with $0<k'<k$. 
Thus this path cannot return to $A_0$, a contradiction.
\medskip 

(ii) The condition $a\geq \frac{3}{2}b> 0$ was chosen so that the situation is exactly the same as previously. 
Again, $S_a$ splits into 2 arcs. 
On one of them the vertices
have no neighbours, and on the other arc after two steps we always get strictly closer to the origin.
\end{proof}

\begin{center}
\definecolor{qqqqff}{rgb}{0.,0.,1.}
\begin{tikzpicture}
\draw[color=black] (0pt,-10pt) node[right] {\footnotesize $0$};
\draw (-2,-2)-- (3,3);
\draw (-4,4)-- (2,-2);
\draw(-1.,1.) circle (1.4142135623730951cm);
\draw(-1.5,1.5) circle (2.12cm);
\draw[-,color=black] (-2.41,-2) -- (-2.41,4.);
\draw[-,color=black] (0.41,-2) -- (0.41,4.);
\draw[-,color=black] (-4,2.41) -- (3,2.41);
\draw[-,color=black] (-4,-0.41) -- (3,-0.41);
\end{tikzpicture}\\
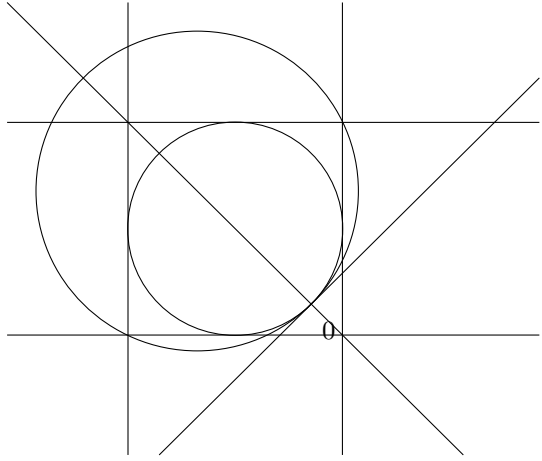
{\sc Figure:} Case (ii)
\end{center}

\begin{example}
On the figure below none of the conditions (i) or (ii) is satisfied.
We do not know whether cycles can exist in this case. 
The larger cicle has now four arcs that could meet cycles.
\end{example}

\begin{center}
\definecolor{qqqqff}{rgb}{0.,0.,1.}
\begin{tikzpicture}
\draw[color=black] (0pt,-10pt) node[right] {\footnotesize $0$};
\draw (-2,-2)-- (3,3);
\draw (-4,4)-- (2,-2);
\draw(-1.1,1.1) circle (1.5556cm);
\draw(-1.3,1.3) circle (1.83cm);
\draw[-,color=black] (-2.655,-2) -- (-2.655,4.);
\draw[-,color=black] (0.455,-2) -- (0.455,4.);
\draw[-,color=black] (-4,2.655) -- (3,2.655);
\draw[-,color=black] (-4,-0.455) -- (3,-0.455);
\end{tikzpicture}\\
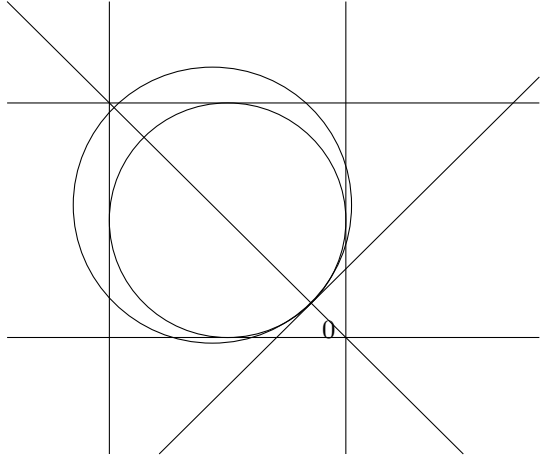
{\sc Figure:} A case where none of (i) and (ii) is satisfied 
\end{center}


\section{Rectangular plates}\label{s5}

We consider the vibrations of a rectangular plate with periodic boundary conditions.
More precisely, we consider the following system in $\Omega=(0,2\pi)\times(0,2\pi)$:

\begin{equation}\label{51}
\begin{cases}
u_{tt}+\Delta^2u=0&\mbox{in }\RR\times\Omega,\\
u(0,\cdot)=u_0&\mbox{in }\Omega,\\
u_t(0,\cdot)=u_1&\mbox{in }\Omega,\\
u(t,x,0)=u(t,x,2\pi)&\mbox{for }t\in\RR\mbox{ and }x\in (0,2\pi),\\
u_y(t,x,0)=u_y(t,x,2\pi)&\mbox{for }t\in\RR\mbox{ and }x\in (0,2\pi),\\
u(t,0,y)=u(t,2\pi,y)&\mbox{for }t\in\RR\mbox{ and }y\in (0,2\pi),\\
u_y(t,0,y)=u_y(t,2\pi,y)&\mbox{for }t\in\RR\mbox{ and }y\in (0,2\pi).
\end{cases}
\end{equation}
The results of this section may be extended to general rectangular domains by a simple linear change of variables.

Let us consider the orthonormal basis
\begin{equation*}
e_{k,\ell}(x,y)=\frac{1}{2\pi}e^{i(kx+\ell y)},
\quad (k,\ell)\in\ZZ^2
\end{equation*}
of $L^2(\Omega)$, and for any fixed real number $s$ let $D^s$ be the Hilbert space obtained by completion of the linear span of the functions $e_{k,\ell}$ with respect to the Euclidean norm
\begin{equation*}
\norm{\sum_{k,\ell}c_{k,\ell}e_{k,\ell}}_s
:=\left(\sum_{k,\ell}(1+k^2+\ell^2)^s\abs{c_{k,\ell}}^2\right)^{1/2}.
\end{equation*}

For any given initial data $u_0\in D^s$ and $u_1\in D^{s-2}$ there is a unique weak solution 
\begin{equation*}
u\in  C(\RR,D^s)\cap C^1(\RR,D^{s-2}).
\end{equation*}
Furthermore, $u$ has a Fourier series representation
\begin{equation}\label{52}
u(t,x)=c_{0,0}^++c_{0,0}^-t+\sum_{\substack{k,\ell\in\ZZ\\(k,\ell)\ne (0,0)}}\left(c_{k,\ell}^+e^{i[(k^2+\ell^2)t+kx+\ell y]}+c{k,\ell}^-e^{i[-(k^2+\ell^2)t+kx+\ell y]}\right)
\end{equation}
with suitable square summable complex coefficients $c_{k,\ell}^+, c_{k,\ell}^-$ satisfying the relations 
\begin{equation*}
\sum_{k,\ell\in\ZZ}(1+k^2+\ell^2)^s(\abs{c_{k,\ell}^+}^2+\abs{c_{k,\ell}^-}^2)\asymp \norm{u_0}_s^2+\norm{u_1}_{s-2}^2.
\end{equation*}
Using \eqref{42} we extend the solutions to $\RR^3$ (by $2\pi$-periodicity in $x$ and $y$).

We are interested in the observability of the solutions on a fixed segment of the form
\begin{equation*}
\set{(x_1+as,y_1+bs)\ :\ s\in (0,1)}
\end{equation*}
during some time interval $(t_1,t_1+T)$
with given real numbers $x_1, y_1, a,b,t_1,T$ satisfying $(a,b)\ne(0,0)$ and $T>0$.

\begin{theorem}\label{t51}
Given any $(a,b)\in\ZZ^2\setminus\set{(0,0)}$ and a real number $T>0$, the solutions of \eqref{51}
satisfy the estimates
\begin{equation*}
\norm{u_0}_0^2+\norm{u_1}_{-2}^2
\ll \int_0^T\int_0^1\abs{u(t_1+t,x_1+as, y_1+bs)}^2\,\mbox{d}s\,\mbox{d}t
\end{equation*}
for all $(u_0,u_1)\in D^1\times D^{-1}$.
\end{theorem}

\begin{remark}\label{r52}
If $a=0$ or $b=0$, then the theorem holds for any nonzero value of the other coefficient by a reasoning similar to the proof of \cite[Theorem 1.3]{Komornik-Loreti-2014}.
\end{remark}

\begin{proof}
It is classical (see, e.g., \cite{Komornik-1994} and the standard trace theorems for Sobolev spaces) that if $(u_0,u_1)\in D^1\times D^{-1}$, then 
the right hand side is well defined.
We  may write $u(t_1+t,x_1+as, y_1+bs)$ in the form
\begin{multline*}
u(t_1+t,x_1+as, y_1+bs)
=d_{0,0}^++d_{0,0}^-t\\
+\sum_{\substack{k,\ell\in\ZZ\\(k,\ell)\ne (0,0)}}\left(d_{k,\ell}^+e^{i[(k^2+\ell^2)t+(ak+b\ell)s]}+d_{k,\ell}^-e^{i[-(k^2+\ell^2)t+(ak+b\ell)s]}\right)
\end{multline*}
with suitable complex coefficients satisfying the relations 
\begin{equation*}
\abs{d_{0,0}^+}^2+\abs{d_{0,0}^-}^2
\asymp \abs{c_{0,0}^+}^2+\abs{c_{0,0}^-}^2
\end{equation*}
and the equalities $\abs{d_{k,\ell}^\pm}=\abs{c_{k,\ell}^\pm}$ for all $(k,\ell)\ne(0,0)$.
Therefore the theorem will follow if we show that the set 
\begin{equation*}
\set{(ak+b\ell,\pm[k^2+\ell^2])\ :\ k,\ell\in\ZZ}
\end{equation*}
is uniformly separated.
Setting 
\begin{equation*}
m:=ak+b\ell,\quad n:=bk-a\ell
\end{equation*}
we have $(m,n)\in\ZZ^2$ and
\begin{equation*}
k^2+\ell^2=\frac{m^2+n^2}{a^2+b^2}.
\end{equation*}
Noticing that the linear map $(k,\ell)\mapsto (m,n)$ is invertible, it suffices to show that the set 
\begin{equation*}
\set{(m,\pm\frac{m^2+n^2}{a^2+b^2})\ :\ m,n\in\ZZ},
\end{equation*}
is uniformly separated.
This is equivalent to  the uniform separatedness of
\begin{equation*}
\set{(m,m^2+n^2)\ :\ m,n\in\ZZ},
\end{equation*}
and this was proved in \cite{Tenenbaum-Tucsnak-2009}.
\end{proof}

\section{Open problems}

We end this paper by a list of open questions:

\begin{enumerate}[\upshape (i)]
\item We do not know whether the inverse inequality in Theorem \ref{t32} (ii) fails for every choice of finitely many segments $(x_i,t_i)$ instead of only one segment.
\item Can we consider more general segments in Proposition \ref{p34} (ii)?
\item How to modify Proposition \ref{p34}
for Neumann boundary conditions?
\item We have no examples for a cycle in Theorem \ref{t48}.
Does the inverse inequality \eqref{48} of Theorem \ref{t48} hold for any two different non-integer rationals $a_1, a_2$?
\item If there are examples of cycles in the preceding question, then does there exist an integer $m\ge 3$ such that the inverse inequalities
\begin{equation*}
\sum_{k\in\ZZ}\left(\abs{c_k^+}^2+\abs{c_k^-}^2\right)\ll 
\sum_{i=1}^m\int_0^T\abs{u(t_i+t,x_i-a_it)}^2\,\mbox{d}t
\end{equation*}
hold for all choices of different nonzero integers $a_1,\ldots,a_m$?
\end{enumerate}

\end{document}